\documentclass[11pt]{article}

\usepackage[margin=1in,letterpaper]{geometry}
\usepackage[parfill]{parskip}
\usepackage[utf8]{inputenc}
\usepackage{physics}
\usepackage{bbm} 
\usepackage{color}
\usepackage{float}
\usepackage{hyperref} 
\usepackage{cleveref} 
\usepackage{xargs}
\usepackage[dvipsnames]{xcolor}
\usepackage{algorithmicx}
\usepackage{graphics}
\usepackage{graphicx}
\usepackage{pstricks}
\usepackage{pst-node}
\usepackage{pst-tree}
\usepackage{tikz} 
\usepackage{enumitem}
\setlistdepth{9}
\newlist{myEnumerate}{enumerate}{9}
\setlist[myEnumerate,1,2,3,4,5,6,7,8,9]{label*=\arabic*.} 
\renewlist{itemize}{itemize}{9}
\setlist[itemize]{label=\textbullet}
\setlist[itemize,2]{label=--}
\setlist[itemize,3]{label=*}
\setlist[itemize,5]{label=--}
\setlist[itemize,6]{label=*}
\setlist[itemize,8]{label=--}

\usepackage{titlesec}

\titleformat{\paragraph}
{\normalfont\normalsize\bfseries}{\theparagraph}{1em}{}
\titlespacing*{\paragraph}
{0pt}{3.25ex plus 1ex minus .2ex}{1.5ex plus .2ex}

\crefname{paragraph}{section}{sections}

\usepackage{amsmath,amssymb,amsfonts,amsthm}
\usepackage{blkarray} 

\makeatletter
\DeclareRobustCommand{\bigDelta}{\mathop{\vphantom{\sum}\mathpalette\bigDelta@\relax}\slimits@}
\newcommand{\bigDelta@}[2]{\vcenter{\sbox\z@{$#1\sum$}\hbox{\resizebox{.9\dimexpr\ht\z@+\dp\z@}{!}{$\m@th\Delta$}}}}
\makeatother

\newcommand{\VM}{$\ensuremath{\mathrm{ISO}\,\text{-}\,\mathrm{VERTEXMINOR}}$}
\newcommand{\SVM}{$\ensuremath{\mathrm{ISO}\,\text{-}\,\mathrm{STARVERTEXMINOR}}$}
\newcommand{\SOET}{$\ensuremath{\mathrm{ISO}\,\text{-}\,\mathrm{SOET}}$}
\newcommand{\MINOR}{$\ensuremath{MINOR}$}
\newcommand{\HMINOR}{$H-\ensuremath{MINOR}$}

\newcommand{\CH}{$\ensuremath{\mathrm{CUBHAM}}$}

\newcommand{\bs}[1]{\ensuremath{\boldsymbol{#1}}}

\usepackage{algpseudocode} 
\usepackage{algorithm} 

\newtheorem{theorem}{Theorem}[section]
\newenvironment{thm}{$\vspace{-0.5em}$\begin{theorem}}{\hfill$\diamond$\end{theorem}}
\crefname{thm}{theorem}{theorems}

\newtheorem{corollary}{Corollary}[theorem]

\crefname{cor}{corollary}{corollaries}

\newtheorem{lemma}{Lemma}[section]
\newenvironment{lem}{$\vspace{-0.5em}$\begin{lemma}}{\hfill$\diamond$\end{lemma}}
\crefname{lem}{lemma}{lemmas}

\newtheorem{definition}{Definition}[section]
\newenvironment{mydef}{$\vspace{-0.5em}$\begin{definition}}{\hfill$\diamond$\end{definition}}
\crefname{mydef}{definition}{definitions}

\newsavebox{\mybox}
\newtheorem{problem}{Problem}[section]
\newenvironment{prm}{$\vspace{-0.5em}$\begin{problem}}{\hfill$\diamond$\end{problem}}
\crefname{problem}{problem}{problems}

\usepackage{subcaption}
\usepackage[colorinlistoftodos,prependcaption]{todonotes}
\newcommandx{\unsure}[2][1=]{\todo[linecolor=red,backgroundcolor=red!25,bordercolor=red,#1]{#2}}
\newcommandx{\change}[2][1=]{\todo[linecolor=blue,backgroundcolor=blue!25,bordercolor=blue,#1]{#2}}
\newcommandx{\info}[2][1=]{\todo[linecolor=OliveGreen,backgroundcolor=OliveGreen!25,bordercolor=OliveGreen,#1]{#2}}
\newcommandx{\improvement}[2][1=]{\todo[linecolor=Plum,backgroundcolor=Plum!25,bordercolor=Plum,#1]{#2}}
\newcommandx{\thiswillnotshow}[2][1=]{\todo[disable,#1]{#2}}

\title{The complexity of the vertex-minor problem}
\author{Axel Dahlberg \and Jonas Helsen \and Stephanie Wehner}
\date{QuTech - TU Delft, Lorentzweg 1, 2628CJ Delft, The Netherlands\\[2ex]
\today }
\begin{document}

\maketitle
\begin{abstract}
A graph $H$ is a vertex-minor of a graph $G$ if it can be reached from $G$ by the successive application of local complementations and vertex deletions.
Vertex-minors have been the subject of intense study in graph theory over the last decades and have found applications in other fields such as quantum information theory.
Therefore it is natural to consider the computational complexity of deciding whether a given graph $G$ has a vertex-minor isomorphic to another graph $H$, which was previously unknown.
Here we prove that this decision problem is NP-complete, even when restricting $H,G$ to be circle graphs, a class of graphs that has a natural relation to vertex-minors.
\end{abstract}

\section{Introduction}

A central problem in graph theory is the study of `substructures' of graphs.
These substructures of some graph $G$ are usually defined as the graphs which can be reached, from $G$, by a given set of graph operations.
An example of such a substructure to be seriously studied early on, was the graph minor, where the central question was to decide whether a graph $G$ can be transformed into a graph $H$ through the successive application on $G$ of vertex deletions, edge deletions and edge contractions~\cite{Wagner1937}.
If this is the case we call $H$ a minor of $G$.
The problem (\MINOR) of deciding whether a graph $H$ is a minor of $G$ is NP-complete when both $H$ and $G$ are part of the input to the problem~\cite{MATOUSEK1992343}.
However, given a fixed $H$, we can define the problem (\HMINOR) of deciding whether $H$ is a minor of $G$, where only $G$ is part of the input.
As expounded in Robertson \& Seymour's seminal series of papers \cite{RobertsonFullSeries}, \HMINOR\ is solvable in cubic time for any graph $H$.
Since then a great variety of minor-relations has been defined and for many of those the complexity has been studied.
Of particular interest recently are the minor-relations related to the graph operation of local complementation, namely vertex- and pivot-minors.
These two minor structures have been studied within the graph theory community~\cite{Oum2005,jeong2014excluded,kwon2014graphs,geelen2009circle} but have also found surprising applications outside of it, notably in the field of quantum information science~\cite{Dahlberg2018,hahn2018quantum,mhalla2012graph,duncan2013pivoting,van2004efficient,van2004graphical}.
The complexity of the vertex- and pivot-minor decision problems was a notable open problem (see question 7 in \cite{oum2016structural}).
Recently it was proven in~\cite{dabrowski2016recognizing} that the pivot-minor problem is NP-complete if both $G$ and $H$ are part of the input to the problem, but the complexity of the vertex-minor problem was left open.
In \cite{dahlberg2018transform} we proved, in the context of quantum information theory, the NP-completeness of the labeled version of the vertex-minor problem, i.e. the problem of deciding if $H$ is a vertex-minor of the graph $G$, taking labeling into account.
The labeled version of the vertex-minor problem is relevant in the context of quantum information theory since there the vertices of the graph corresponds to physical qubits in, for example, a quantum network.
However we did not discuss the complexity of the related problem of deciding whether $G$ has a vertex-minor isomorphic to $H$ (on any subset of the vertices).
Here we close this gap, proving that the unlabeled version of the vertex-minor problem is also NP-complete.
To avoid confusion with the problems studied in \cite{dahlberg2018transform} we will here call the unlabeled version of the vertex-minor problem \VM.
Moreover we here prove that \VM\ remains NP-complete even when $H$ is restricted to be a star graph and $G$ a circle graph.
Our work resolves the problem left open in~\cite{dabrowski2016recognizing} and provides a partial answer to the questions posed in \cite{oum2016structural}.
In the process of proving hardness we make use of the concept of the semi-ordered Eulerian tour (Soet), a graph construction we introduced in \cite{dahlberg2018transform} that may be of further independent interest.\\

The paper is organized as follows: in \cref{sec:prel} we recall relevant graph theoretic notions such as vertex-minors and circle graphs. We also discuss the concept of semi-ordered Eulerian tours. In \cref{sec:npcomp} we prove the main result: the NP-completeness of the vertex-minor problem.

\section{Preliminaries}\label{sec:prel}
In this section we review relevant graph theoretical notions. We begin by treating the local complementation operation and the notion of vertex-minors before discussing a class of graphs called circle graphs. Here we also introduce the notion of a semi-ordered Eulerian tour and connect it to the existence of star-graph vertex-minors of circle graphs.\\

We will denote graphs by capital letters: $G,H,F,R,..$. Graphs are assumed to be simple unless otherwise indicated. The vertex-set of a graph $G$ is denoted $V(G)$ and the edge-set is denoted $E(G)$. Give a vertex $v$ in a graph $G$ we call denote the neighborhood of $v$ (the set of vertices adjacent to $v$ in $G$) by $N_v$. Given a graph $G$ and a subset of its vertices $V'$ we will denote the induced subgraph of $G$ on those vertices by $G[V']$.
We denote the fully connected graph on $n$ vertices as $K_n$.

We denote words, i.e. ordered sequences of elements of a set (with repetition) by boldface letters, i.e. $\bs{X},\bs{Y}, \bs{Z},...$. We denote the mirroring (reversing of the ordering of its letters) of a word $\bs{X}$ by $\widetilde{\bs{X}}$.
Throughout this paper we use the following notation for sets of consecutive natural numbers
\begin{align}
    [k, n] &\equiv \{i\in\mathbb{N}: k\leq i< n\} \\
    [n] &\equiv \{i\in\mathbb{N}: 0\leq i< n\}
\end{align}

\subsection{Vertex-minors}
We review the definition of local complementation: 

\begin{mydef}[Local complementation]\label{def:LC_prel}
    A local complementation $\tau_v$ is a graph operation specified by a vertex $v$, taking a graph $G$ to $\tau_v(G)$ by replacing the induced subgraph on the neighborhood of $v$, i.e.  $G[N_v]$, by its complement.
    The neighborhood of any vertex $u$ in the graph $\tau_v(G)$ is therefore given by
    \begin{equation}
    N_u^{(\tau_v(G))}=\begin{cases}N_u\Delta (N_v\setminus\{u\}) & \quad \text{if } (u,v)\in E(G) \\ N_u & \quad \text{else},\end{cases}
    \end{equation}
    where $\Delta$ denotes the symmetric difference between two sets.

\end{mydef}
    Given a sequence of vertices $\bs{v}=v_1\dots v_k$, we denote the induced sequence of local complementations, acting on a graph $G$, as
    \begin{equation}
        \tau_{\bs{v}}(G)=\tau_{v_k}\circ\dots\circ\tau_{v_1}(G).
    \end{equation}

A graph $H$ that can be reached from another graph $G$ using local complementations and vertex-deletions is called a \emph{vertex-minor}~\cite{Oum2005} and is formally defined as:

\begin{mydef}[Vertex-minor]\label{def:vertex-minor_prel}
    A graph $H$ is called a vertex-minor of $G$ if there exist a sequence of local complementations and vertex-deletions that takes $G$ to $H$.
    If $H$ is a vertex-minor of $G$ we write this as
    \begin{equation}
        H<G
    \end{equation}
\end{mydef}
Associated to the notion of vertex-minor is the natural decision problem:

\begin{prm}[\VM]
    Given a graph $G$ and a graph $H$ decide whether their exists a graph $\tilde{H}$ such that (1) $H$ and $\tilde{H}$ are isomorphic, and (2) $\tilde{H}<G$.
\end{prm}
We can restrict this problem to a special case, where the graph $H$ is a star-graph on $k$ vertices. We call this problem \SVM. Note that we must only specify $k$ as there exists only one star-graph on $k$ vertices up to isomorphism. Formally we have

\begin{prm}[\SVM]
    Given a graph $G$ and an integer $k$ decide whether there exists a subset $V'$ of $V(G)$ with $|V'|=k$ and a star graph on $V'$ denoted $S_{V'}$ such that $S_{V'}<G$.
\end{prm}

\subsection{Circle graphs}\label{sec:circle}

Here we review circle graphs and representations of these under the action of local complementations.
Circle graphs are also sometimes called alternance graphs since they can be described, as explained below, by a double-occurrence word such that the edges of the graph are the given by the alternances induced by this word, .
We will make use of this description here, which was introduced by Bouchet in~\cite{Bouchet1972} and also described in~\cite{Bouchet1994}.
This description is also related to yet another way to represent circle graphs, as Eulerian tours of $4$-regular multi-graphs, introduced by Kotzig in~\cite{Kotzig1977}. 
For an overview of the theory and history of circle graphs see for example the book by Golumbic~\cite{Golumbic2004}.

\subsubsection{Double-occurrence words}
Let us first define double-occurrence words and equivalence classes of these. This will allow us to define circle graphs.

\begin{mydef}[Double occurrence word]\label{def:dow}
    A double-occurrence word $\bs{X}$ is a word with letters in some set $V$, such that each element in $V$ occurs exactly twice in $\bs{X}$.  
\end{mydef}

Given a double-occurrence word $\bs{X}$ we will write $V(\bs{X})=V$ for its set of letters.

\begin{mydef}[Equivalence class of double-occurrence words]\label{def:d_X}
    We say that a double-occurrence word $\bs{Y}$ is equivalent to another $\bs{X}$, i.e. $\bs{Y}\sim \bs{X}$, if $\bs{Y}$ is equal to $\bs{X}$, the mirror $\widetilde{\bs{X}} $ or any cyclic permutation of $\bs{X}$ or $\widetilde{\bs{X}}$.
    We denote by $\bs{d}_{\bs{X}}= \{\bs{Y} \::\:\bs{Y}\sim \bs{X}\}$ the equivalence class of $\bs{X}$, i.e. the set of words equivalent to $\bs{X}$.

\end{mydef}

Next we define alternances of these equivalence classes, which will represent the edges of an alternance graph.

\begin{mydef}[Alternance]\label{def:alternance}
    An \emph{alternance} $(u,v)$ of the equivalence class $\bs{d}_{\bs{X}}$ is a pair of distinct elements $u,v\in V$ such that a double-occurrence word of the form $\dots u\dots v\dots u\dots v\dots$ is in $\bs{d}_{\bs{X}}$.
\end{mydef}

Note that if $(u,v)$ is an alternance of $\bs{d}_{\bs{X}}$ then so is $(v,u)$, since the mirror of any word in $\bs{d}_{\bs{X}}$ is also in $\bs{d}_{\bs{X}}$.

\begin{mydef}[Alternance graph]\label{def:alternance_graph}
    The alternance graph $\mathcal{A}(\bs{X})$ of a double-occurrence word $\bs{X}$ is a graph with vertices $V(\bs{X})$ and edges given exactly by the alternances of $\bs{d}_{\bs{X}}$, i.e.
    \begin{equation}
        E(\mathcal{A}(\bs{X}))=\{(u,v)\in V(\bs{X})\times V(\bs{X})\,:\,(u,v)\text{ is an alternance of }\bs{d}_{\bs{X}}\}.
    \end{equation}
\end{mydef}

Note that since $\mathcal{A}(\bs{X})$ only depends on the equivalence class of $\bs{X}$, the alternance graphs $\mathcal{A}(\bs{X})$ and $\mathcal{A}(\bs{Y})$ are equal if $\bs{X}\sim\bs{Y}$.
Now we can formally define circle graphs.

\begin{mydef}[Circle graph]\label{def:circle_graph}
    A graph $G$ which is the alternance graph of some double-occurrence word $\bs{X}$ is called a circle graph.
\end{mydef}

As an example, consider the following double-occurrence word with letters in the set $V_0=\{a, b, c, d, e\}$:
\begin{equation}\label{eq:example}
    \bs{X}_0 = adcbaebced
\end{equation}
The alternances of $d_{\bs{X}_0}$ are thus
\begin{equation}
    (a,b),\; (a,c),\; (a,d),\; (b,e),\; (c, e)
\end{equation}
and their mirrors. The alternance graph $\mathcal{A}(\bs{X}_0)$ is therefore the graph in \cref{fig:example}.

\begin{figure}[H]
    \centering
    \includegraphics[width=0.25\textwidth]{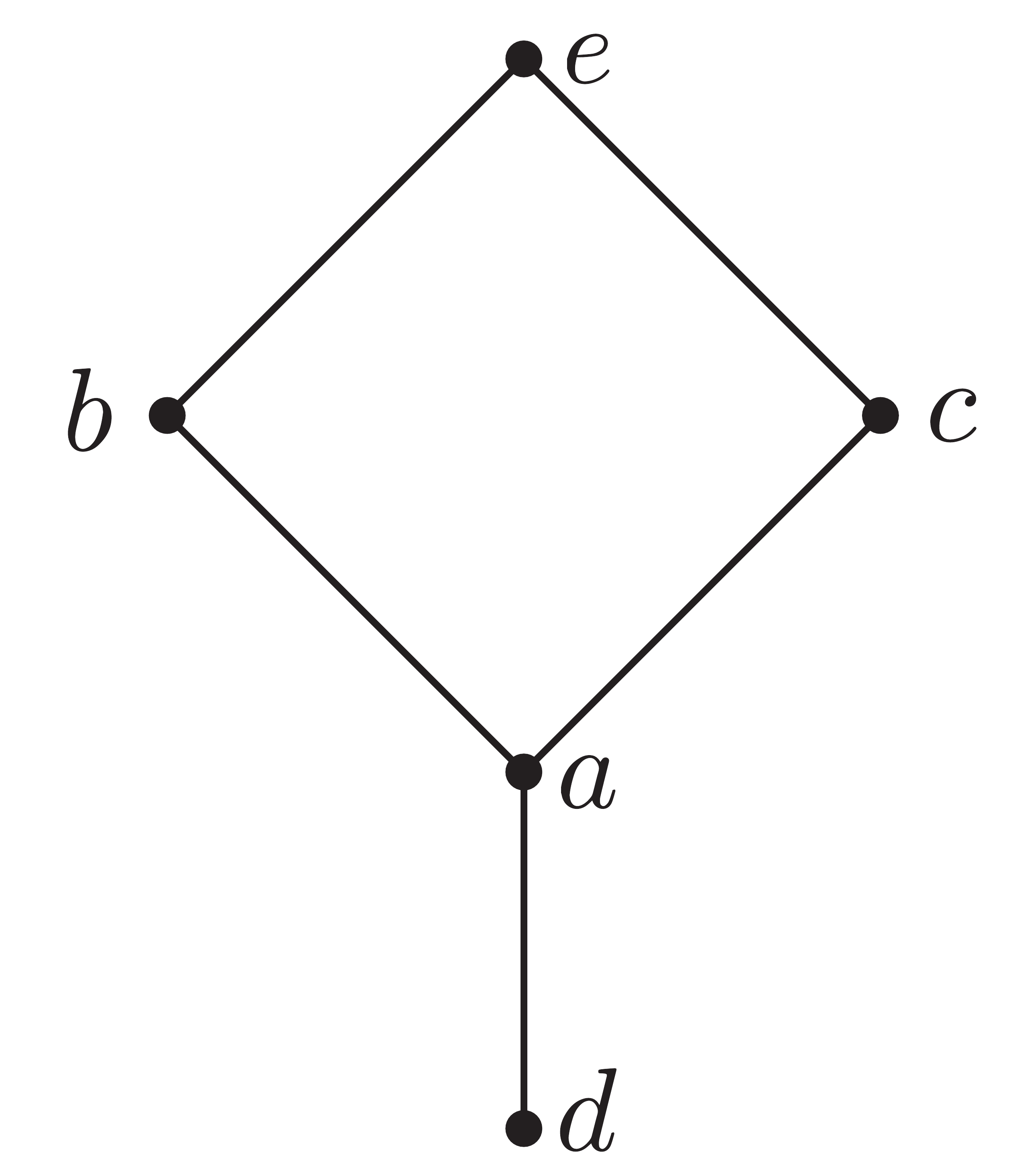}
    \caption{An example of a circle graph induced by the double-occurrence word $adcbaebced$.}
    \label{fig:example}
\end{figure}

\subsubsection{Eulerian tours on 4-regular multi-graphs}
There is yet another way to represent circle graphs, closely related to double-occurrence words, namely as Eulerian tours of $4$-regular multi-graphs.
A $4$-regular multi-graph is a graph where each vertex has exactly four incident edges and can contain multiple edges between each pair of vertices or edges only incident to a single vertex.

\begin{mydef}[Eulerian tour]
    Let $F$ be a connected $4$-regular multi-graph.
    An Eulerian tour $U$ on $F$ is a tour that visits each edge in $F$ exactly once.
\end{mydef}

Any $4$-regular multi-graph is Eulerian, i.e. has a Eulerian tour, since each vertex has even degree~\cite{biggs1976graph}.

Furthermore, any Eulerian tour on a $4$-regular multi-graph $F$ traverses each vertex exactly twice, except for the vertex which is both the start and the end of the tour.
Such a Eulerian tour induces therefore a double-occurrence word, the letters of which are the vertices of $F$, and consequently a circle graph as described in the following definition.

\begin{mydef}[Induced double-occurrence word]\label{def:eul_tour}
    Let $F$ be a connected $4$-regular multi-graph on $k$ vertices $V(F)$.
    Let $U$ be a Eulerian tour on $F$ of the form
    \begin{equation}\label{eq:eul_tour}
        U=x_0e_0x_1\dots x_{2k-2}e_{2k-2}x_{2k-1}e_{2k-1}x_0.
    \end{equation}
    with $x_i\in V(F)$ and $e_i\in E(F)$. Note that every element of $V$ occurs exactly twice in $U$, except $x_0$.
    From a Eulerian tour $U$ as in \cref{eq:eul_tour} we define an induced double-occurrence word as
    \begin{equation}
        m(U)=x_1x_2\dots x_{2k-1}x_{2k}.
    \end{equation}
    To denote the alternance graph given by the double-occurrence word induced by a Eulerian tour, we will write $\mathcal{A}(U)\equiv\mathcal{A}(m(U))$.
\end{mydef}

Similarly to double-occurrence words, we also introduce equivalence classes of Eulerian tours under cyclic permutation or reversal of the tour.

\begin{mydef}[Equivalence class of Eulerian tours]
    Let $F$ be a connected $4$-regular multi-graph and $U$ be an Eulerian tour on $F$.
    We say that an Eulerian tour $U'$ on $F$ is equivalent to $U$, i.e. $U\sim U'$, if $U'$ is equal to $U$, the reversal $\widetilde{U}$ or any cyclic permutation of $U$ or $\widetilde{U}$.
    We denote by $\bs{t}_U$ the equivalence class of $U$, i.e. the set of Eulerian tours on $F$ which are equivalent to $U$.
\end{mydef}

It is clear that if the Eulerian tours $U$ and $U'$ on a $4$-regular multi-graph $F$ are equivalent, then so are the double-occurrence words $m(U)$ and $m(U')$.
Furthermore, as for double-occurrence words, two equivalent Eulerian tours on a connected $4$-regular multi-graph induce the same alternance graph.

Consider for example the $4$-regular multi-graph in \cref{fig:Soet_ex}(a).
This graph has a tour $U_0$ with an induced double-occurrence word
\begin{equation}
    m(U) = adcbaebced
\end{equation}
Note, that this is equal to the word in \cref{eq:example} which shows that $\mathcal{A}(U)$ is also the graph in \cref{fig:example}.

\subsubsection{Vertex-minors of circle graphs}

When we are considering vertex-minors of circle graphs, it is useful to map the operations of local complementation and vertex deletion on an alternance graph of a double-occurrence word to operations on that double-occurrence word. \\

We start by considering local complementation.
Let $\bs{X}=\bs{A}v\bs{B}v\bs{C}$ be a double-occurrence word with alternance graph $\mathcal{A}(\bs{X})$ and let $v$ be an element in $V(\bs{X})$.
Local complementation at the vertex $v$ in the graph $\mathcal{A}(\bs{X})$ now corresponds to the mirroring of the sub-word $\bs{B}$ of $\bs{X}$ in between the two occurrences of $v$, i.e.
\begin{equation}
\tau_{v}\big(\mathcal{A}(\bs{X})\big) = \mathcal{A}(\bs{A}v\widetilde{\bs{B}}v\bs{C})
\end{equation}
Note that both the double-occurrence word $\bs{X} = \bs{A}v\bs{B}v\bs{C}$ and the double-occurrence word $\bs{A}v\widetilde{\bs{B}}v\bs{C}$ arise as words induced by Eulerian tours on the same $4$-regular graph $F$.
One can in fact show~\cite{Bouchet1994} that two circle graphs are equivalent under the action of local complementation if and only if they arise as alternance graphs induced by Eulerian tours on the same $4$-regular multi-graph.

Next we consider vertex deletion.
We will denote by $\bs{X}\setminus v$ the deletion of the element $v$, i.e.
\begin{equation}\label{eq:VD_on_m}
    \bs{X}\setminus v\equiv (\bs{A}v\bs{B}v\bs{C})\setminus v=\bs{A}\bs{B}\bs{C}.
\end{equation}
The resulting word $\bs{ABC}$ is also a double-occurrence word and furthermore we have that
\begin{equation}
    \mathcal{A}(\bs{X})\setminus v=\mathcal{A}(\bs{X}\setminus v).
\end{equation}
If $W=\{w_1,w_2\dots,w_l\}$ is a subset of $V$, we will write $\bs{X}\setminus W$ as the deletion of all elements in $W$, i.e.
\begin{equation}
    \bs{X}\setminus W=(\dots((\bs{X}\setminus w_1)\setminus w_2)\dots)\setminus w_l.
\end{equation}
Connected to this we can also define an induced double-occurrence sub-word $\bs{X}[W]=\bs{X}\setminus (V\setminus W)$.
The reason for calling this an induced double-occurrence sub-word stems from its relation to induced subgraphs of the alternance graph as
\begin{equation}\label{eq:induced_alternance}
    \mathcal{A}(\bs{X})[W]=\mathcal{A}(\bs{X}[W]).
\end{equation}

We can decide if a circle graph has a certain vertex-minor by considering Eulerian tours of a $4$-regular graph, which is captured in the following theorem, a proof of which can be found in~\cite{dahlberg2018transform}. This theorem states that vertex-minors of alternance graphs induced by a Eulerian tour on a $4$-regular graph $F$ are exactly the alternance graphs induced by sub-words formed by Eulerian tours on $F$.

\begin{thm}\label{thm:vm_of_eul}
    Let $F$ be a connected $4$-regular multi-graph and let $G$ be a circle graph such that $G=\mathcal{A}(U)$ for some Eulerian tour $U$ on $F$.
    Then $G'$ is a vertex-minor of $G$ if and only if there exist a Eulerian tour $U'$ on $F$ such that
    \begin{equation}\label{eq:vm_of_eul}
        G'=\mathcal{A}(m(U')[V(G')]).
    \end{equation}
\end{thm}

\subsubsection{Semi-Ordered Eulerian tours}\label{sec:soet}

From the previous sections we have seen that circle graphs and their vertex-minors can be described by Eulerian tours on connected $4$-regular multi-graphs.
One can thus ask, given a graph $H$, what properties a $4$-regular multi-graph $F$ must possess such that any of its alternance graphs\footnote{Note that if $H<\mathcal{A}(U)$ for some Eulerian tour $U$ on $F$ then $H<\mathcal{A}(U')$ for all Eulerian tours $U'$ on $F$.} has $H$ as a vertex-minor.
We answered this question in~\cite{Dahlberg2018} for the case when $H$ is a star graph by introducing the notion of a Semi-ordered Eulerian Tour (Soet), defined as

\begin{mydef}[Soet]\label{def:Soet}
    Let $F$ be a $4$-regular multi-graph and let $V'\subseteq V(F)$ be a subset of its vertices.
    Furthermore, let $\bs{s}=s_0s_1\dots s_{k-1}$ be a word with letters in $V'$ such that each element of $V'$ occurs exactly once in $\bs{s}$ and where $k=\abs{V'}$.
    A semi-ordered Eulerian tour $U$ with respect to $V'$ is a Eulerian tour such that $m(U)=s_0\bs{X}_0s_1\dots s_{k-1}\bs{X}_{k-1}s_0\bs{Y}_{0}s_1\dots s_{k-1}\bs{Y}_{k-1}$ and where $\bs{X}_0,\bs{X}_1,\dots,\bs{X}_{k-1},\bs{Y}_0,\dots,\bs{Y}_{k-1}$ are words (possibly empty) with letters in $V(F)\setminus V'$.
    This can also be stated as $m(U)[V']=\bs{s}\bs{s}$, for a word $\bs{s}$.
\end{mydef}
Note that the multi-graph $F$ is not assumed to be simple, so multi-edges and self-loops are allowed.
A Soet is a Eulerian tour on $F$ that traverses the elements of $V'$ in some order once and then again in the same order.
The particular order in which the Soet~traverses $V'$ will not be important here, only that it traverses $V'$ in the same order twice.
As an example, the graph in \cref{fig:Soet_ex}(a) allows for a Soet with respect to the set $\{a,b,c,d\}$ but the graph in \cref{fig:Soet_ex}(b) does not.

\begin{figure}[H]
    \centering
    \begin{subfigure}{0.35\textwidth}
        \includegraphics[width=1\textwidth]{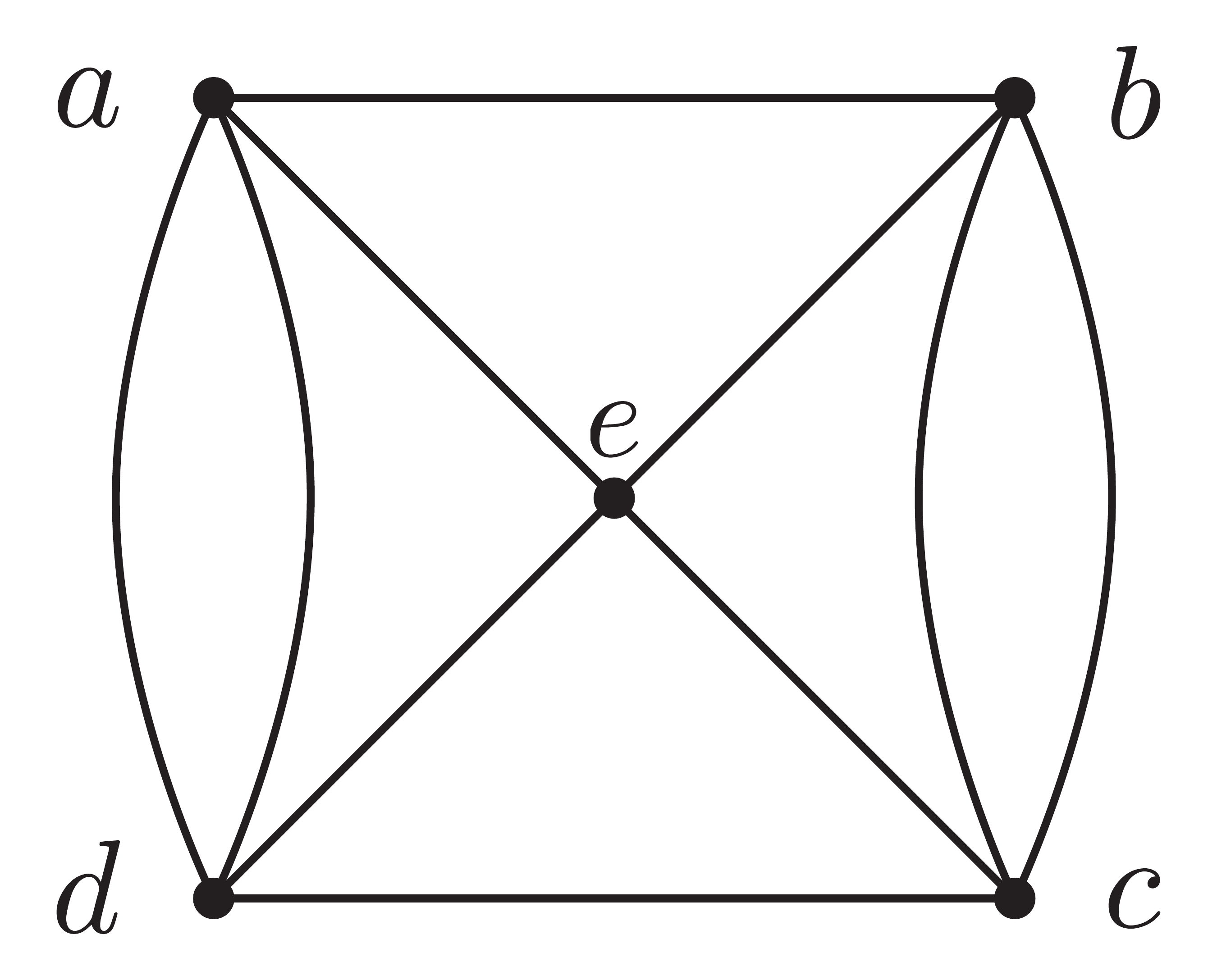}
        \caption{}
        \label{fig:Soet_yes}
    \end{subfigure}
    \hspace{2cm}
    \begin{subfigure}{0.35\textwidth}
        \includegraphics[width=1\textwidth]{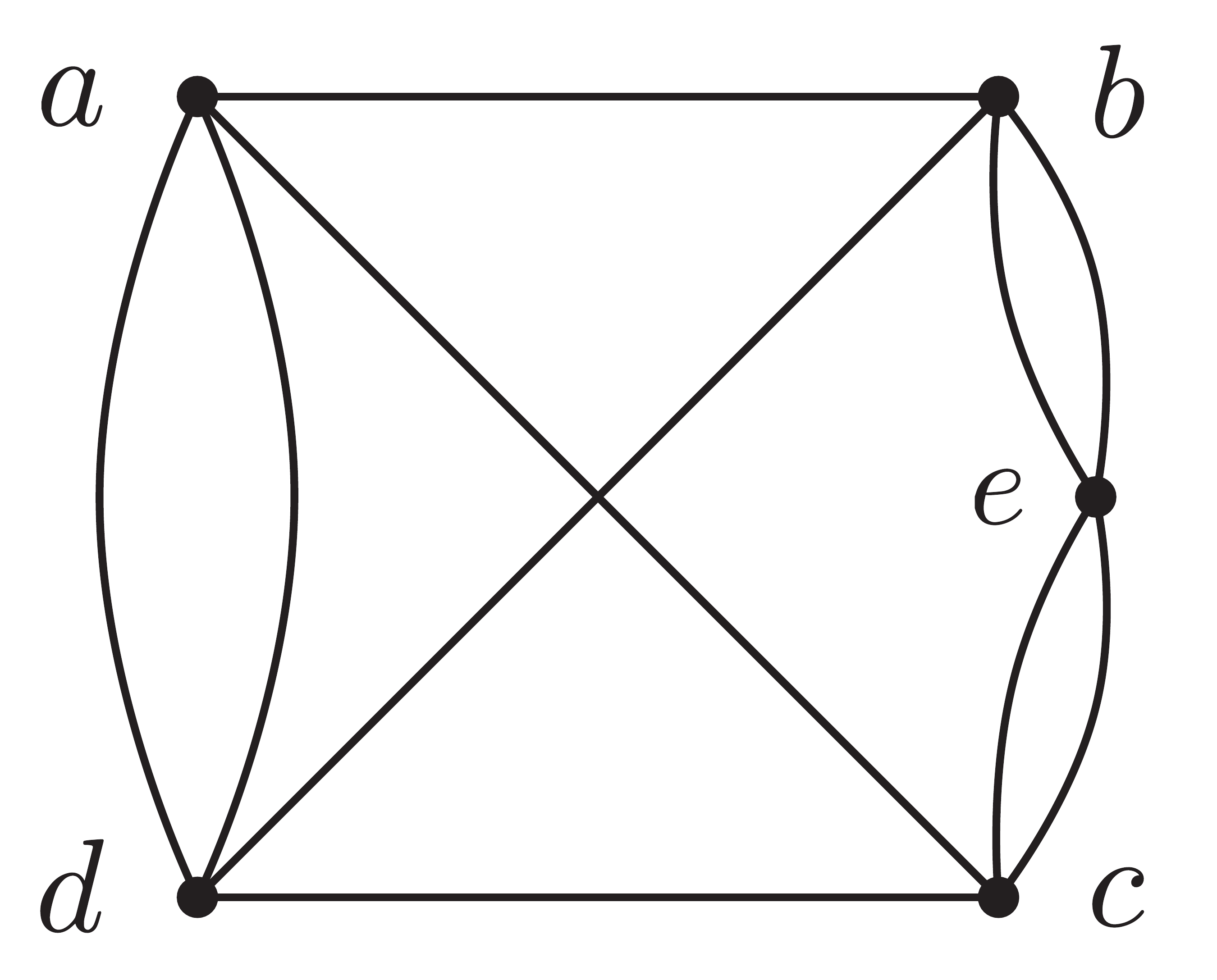}
        \caption{}
        \label{fig:Soet_no}
    \end{subfigure}
    \caption{
        Examples of two $4$-regular multi-graphs.
        \Cref{fig:Soet_yes} is an example of a graph that allows for a Soet with respect to the set $V'=\{a,b,c,d\}$.
        A Soet\ for this graph is for example $m(U)=abcdaebced$.
        The graph in \cref{fig:Soet_no} on the other hand does not allow for any Soet with respect to the set $V'=\{a,b,c,d\}$.
    }
    \label{fig:Soet_ex}
\end{figure}
 
The following theorem, a trivial corollary of Corollary 2.6.1 in~\cite{Dahlberg2018}, connects the problem of finding star graphs as vertex-minors of circle graphs to the problem of finding Soets on $4$-regular multi-graphs.

\begin{thm}\label{cor:reg_to_circle}
    Let $F$ be a connected $4$-regular multi-graph and let $G$ be a circle graph given by the alternance graph of a Eulerian tour $U$ on $F$, i.e. $G=\mathcal{A}(U)$.
    Furthermore let $S_{V'}$ be a star graph on the vertices $V'$.
    Then $S_{V'}<G$ if and only if $F$ allows for a Soet~(see \cref{def:Soet}) with respect to $V'$.
\end{thm}

This gives rise to a natural decision problem which we denote \SOET:

\begin{prm}[\SOET]\label{prob:Soet}
Let $F$ be a $4$-regular multi-graph and let $k\leq |V(F)|$ be an integer. Decide whether there is a $V'\subset V(F)$ with $|V'|=k$ such that there exists a Soet~$U$ on $F$ with respect to the set $V'$.
\end{prm}

In~\cite{dahlberg2018transform} we proved that a version of \cref{prob:Soet} where $V'$ is part of the input to the problem, is NP-complete.
In the next section we prove that also the problem of deciding whether such a $V'$ exists, i.e. \cref{prob:Soet}, is also NP-complete.

One can see that a Soet on a $4$-regular multi-graph $F$ with respect to $V'$, imparts an ordering on the subset of vertices $V'$. We will in particular be interested in vertices in $V'$ that are `consecutive' with respect to the Soet. Consecutiveness is defined as follows.

\begin{mydef}[Consecutive vertices]\label{def:consecutive}
    Let $F$ be a $4$-regular graph and $U$ a~Soet~on $F$ with respect to a subset $V'\subseteq V(F)$.
    Two vertices $u,v\in V'$ are called consecutive in $U$ if there exist a sub-word $u\bs{X}v$ or $v\bs{X}u$ of $m(U)$ such that no letter of $\bs{X}$ is in $V'$.
\end{mydef}

We also define the notion of a ``maximal sub-word" associated with two consecutive vertices.

\begin{mydef}[Maximal sub-words]\label{def:maximal}
    Let $F$ be a $4$-regular multi-graph and $U$ a Soet on $F$ with respect to a subset $V'\subseteq V(F)$.
    The double-occurrence word induced by $U$ is then of the form $m(U)=s_0\bs{X}_0s_1\dots s_{k-1}\bs{X}_{k-1}s_0\bs{Y}_{0}s_1\dots s_{k-1}\bs{Y}_{k-1}$, where $k=\abs{V'}$, $\{s_0,\dots,s_{k-1}\}=V'$ and $\bs{X}_0,\dots,\bs{X}_{k-1},
    \bs{Y}_0,\dots,\bs{Y}_{k-1}$ are words (possibly empty) with letters in $V(F)\setminus V'$.\\
    For $i\in[k]$, we call $\bs{X}_i$ and $\bs{Y}_i$ the two maximal sub-words associated with the consecutive vertices $s_{i}$ and $s_{(i+1\pmod k)}$.
\end{mydef}
    Given two consecutive vertices $u$ and $v$, we will denote their two maximal sub-words as $\bs{X}$ and $\bs{X}'$, $\bs{Y}$ and $\bs{Y}'$ or similar.

 \section{NP-completeness of the vertex-minor problem}\label{sec:npcomp}
In this section we prove the NP-completeness of the vertex-minor problem. This we do in three steps.
We will begin by (1) proving that \SOET\ is NP-Hard. We do this by reducing the problem of deciding whether a $3$-regular graph $R$ is Hamiltonian to \SOET. Next we (2) reduce \SOET\ to \SVM\ and \SVM\ to \VM, thus proving the NP-hardness of \VM. Finally we (3) show that \VM\ is also in NP.

\subsection{SOET is NP-hard}
We first review the definition of a Hamiltonian graph and the associated \CH\ decision problem.

 \begin{mydef}[Hamiltonian]\label{def:hamtour}
    A graph is said to be Hamiltonian if it contains a Hamiltonian cycle.
    A Hamiltonian cycle is a cycle that visits each vertex in the graph exactly once.
\end{mydef}
\begin{prm}[\CH]\label{prob:cubham}
    Let $R$ be a $3$-regular graph. Decide whether $R$ is Hamiltonian.
\end{prm}

The reduction of \CH~to \SOET\ is done by going though the following steps.
\begin{enumerate}
    \item Introduce the notion of a ($4$-regular) $K_3$-expansion $\Lambda(R)$ of a $3$-regular graph $R$. This is done in \cref{def:triangular}.

    \item Prove that if a $3$-regular graph $R$ is Hamiltonian then the $K_3$-expansion $\Lambda(R)$ of $R$ allows for a Soet~of size $2|V(R)|$. This is done in \cref{lem:ham_to_soet}
    \item Prove that if the $K_3$-expansion $\Lambda(R)$ of a $3$-regular graph $R$ allows for a Soet\ of size $2|V(R)|$ then $R$ is Hamiltonian. This is done in \cref{lem:soet_to_ham}
\end{enumerate} 

Note that 1. and 2. above provides necessary and sufficient condition for whether a $3$-regular graph $R$ is Hamiltonian in terms of whether $\Lambda(R)$ allows for a Soet of a certain size.
This implies that \CH\ reduces to \SOET\ and hence that \SOET\ is NP-hard.

We begin by introducing the $K_3$-expansion: a mapping from $3$-regular graphs to $4$-regular multi-graphs.

\begin{mydef}[$K_3$-expansion]\label{def:triangular}
    Let $R$ be a $3$-regular graph.
    A $K_3$-expansion $\Lambda(R)$ of a $3$-regular graph $R$ is constructed from $R$ by performing the following two steps:
    \begin{enumerate}
        \item Replace each vertex $v$ in $R$ with a subgraph isomorphic to $K_3$  as below
            \begin{equation}
                \raisebox{-0.13\textwidth}{\includegraphics[scale=0.2]{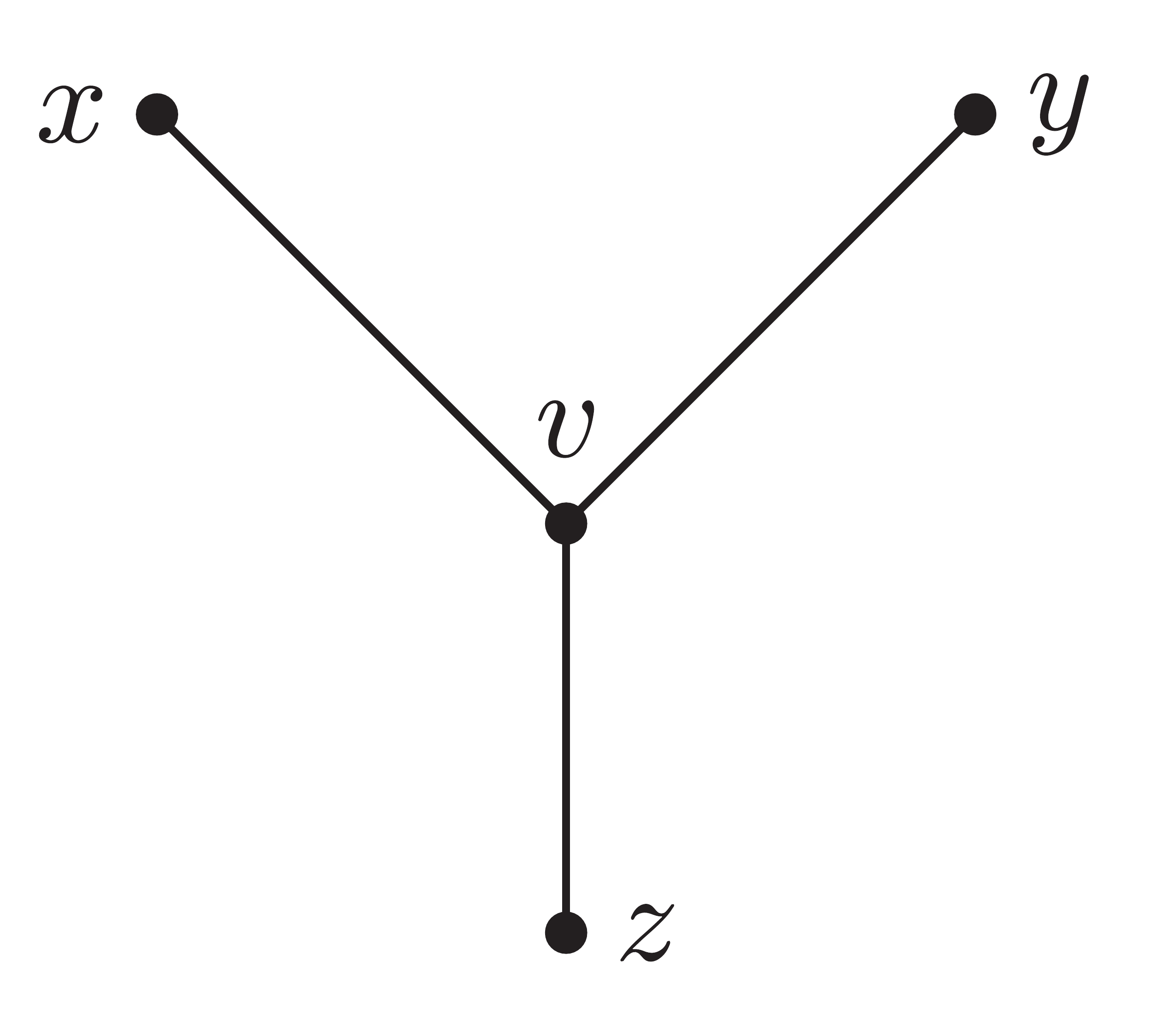}}\quad\implies\quad \raisebox{-0.13\textwidth}{\includegraphics[scale=0.2]{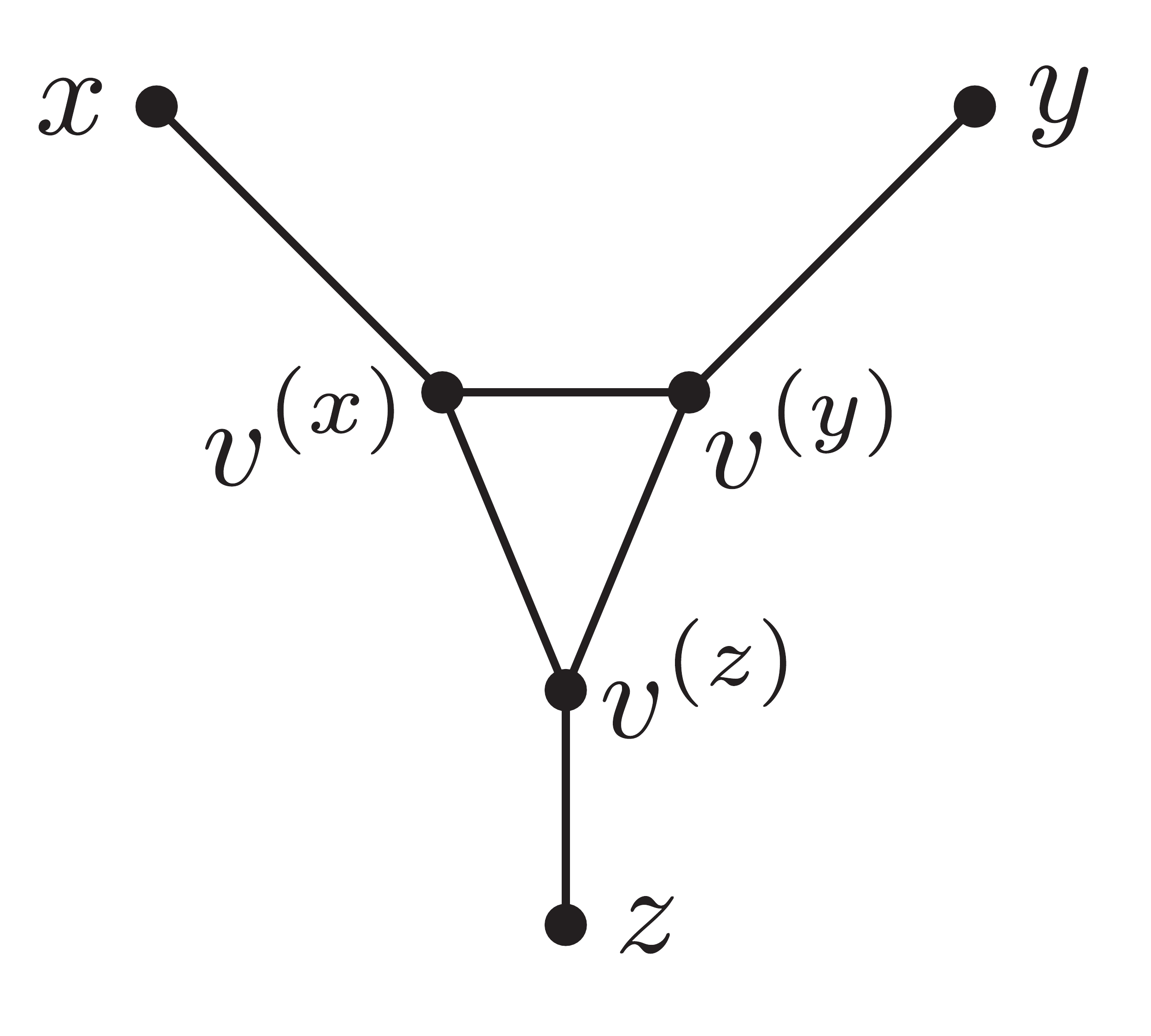}}
            \end{equation}
            
            where $x$, $y$ and $z$ are the neighbors of $v$. We will denote the $K_3$-subgraph associated to the vertex $v$ with $T_v$, i.e. $T_v=G[\{v^{(x)},v^{(y)},v^{(z)}\}]$.
        \item For all $v,v'\in R$ such that $v\neq v'$, double the edge that is incident on two subgraphs $T_v,T_{v'}$.
    \end{enumerate}
\end{mydef}
The graph $\Lambda(R)$ will be called a \emph{$K_3$-expansion} of $R$. A multi-graph $F$ that is the $K_3$-expansion of some $3$-regular graph $R$ will also be referred to as a $K_3$-expanded graph.
Furthermore, the number of vertices in $\Lambda(R)$ is $3\cdot\abs{V(R)}$ and the number of edges is $2\cdot\abs{E(R)}+3\cdot\abs{V(R)}$. In \cref{fig:triangular_expansion_example} we show an example of a $3$-regular graph and its $K_3$-expansion.

\begin{figure}[H]
    \centering
    \begin{subfigure}{0.4\textwidth}
        \includegraphics[width=\textwidth]{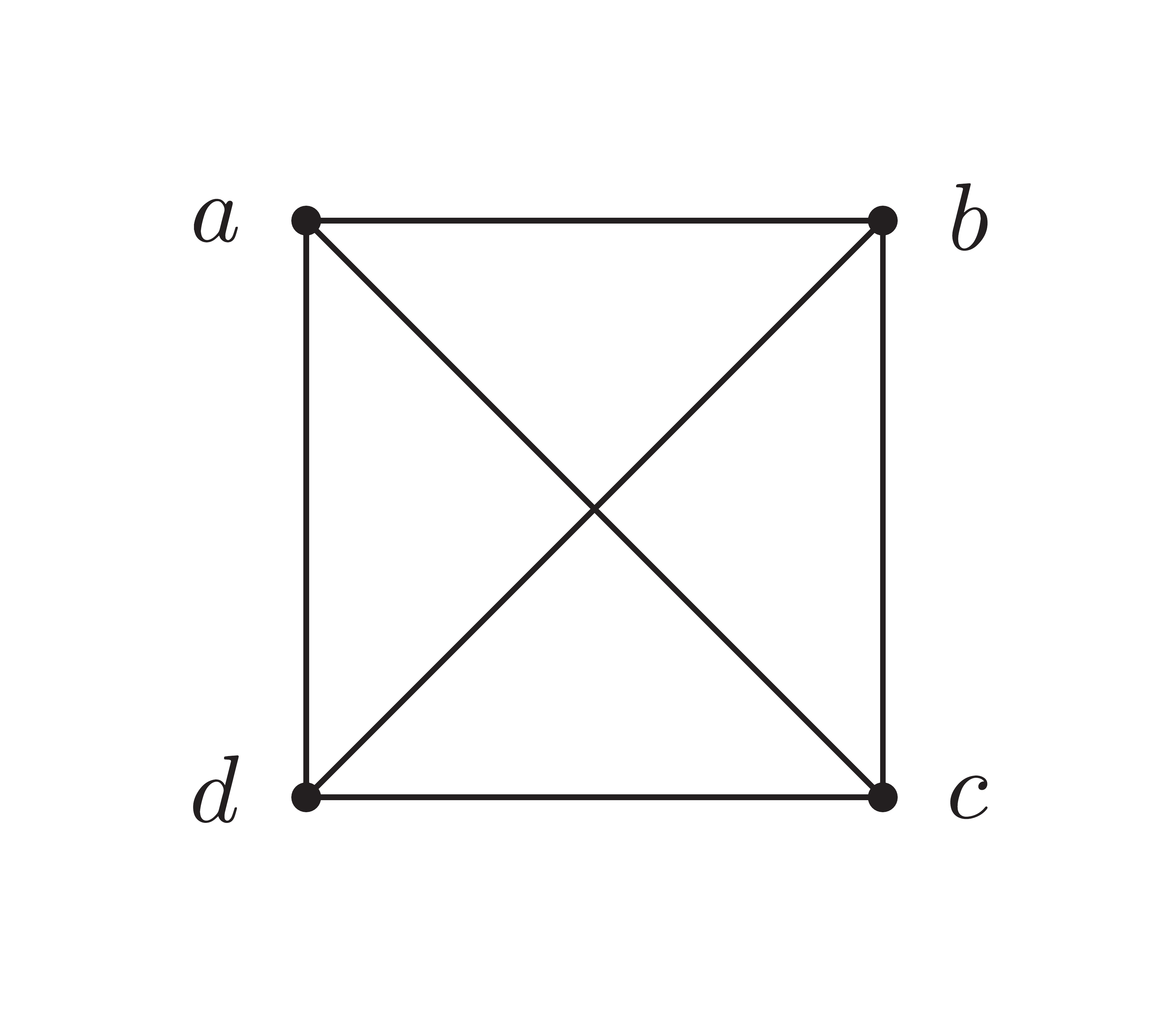}
        \caption{}
    \end{subfigure}
    ~$\quad\implies\quad$
    \begin{subfigure}{0.4\textwidth}
        \includegraphics[width=\textwidth]{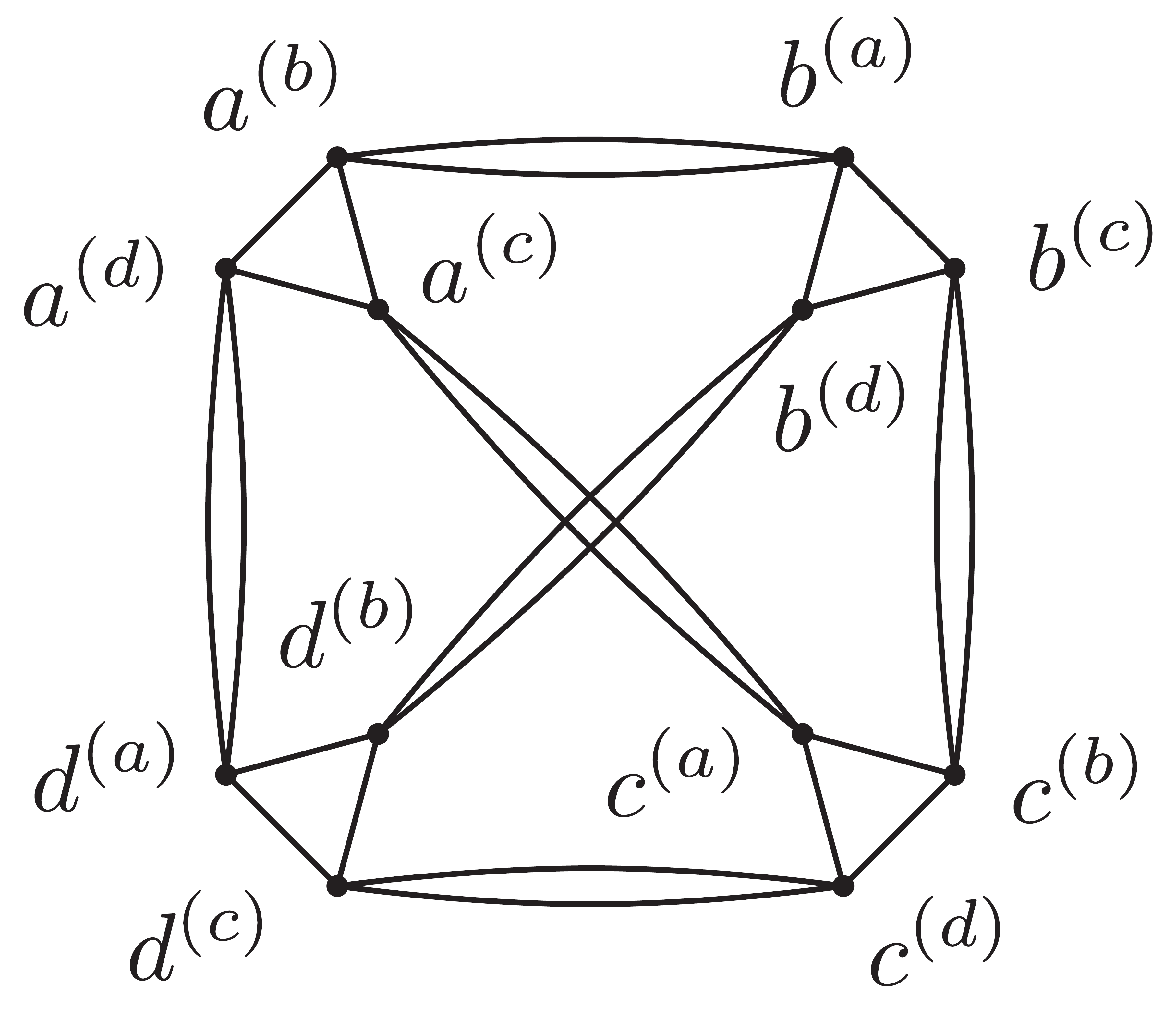}
        \caption{}
    \end{subfigure}
    \caption{Figure showing (a) the complete graph $K_V$ on vertices $V=\{a,b,c,d\}$ and (b) its associated $K_3$-expansion $\Lambda(K_V)$.}
    \label{fig:triangular_expansion_example}
\end{figure}

 We now argue that if a $3$-regular graph $R$ is Hamiltonian then its $K_3$-expansion allows for a Soet on $2|V(R)|$ vertices and thus providing a necessary condition for a $3$-regular graph being Hamiltonian.

\begin{lem}\label{lem:ham_to_soet}
Let $R$ be a triangular graph with $k$ vertices  and let $\Lambda(R)$ be its $K_3$-expansion. If $R$ is Hamiltonian then $\Lambda(R)$ allows for a Soet of size $2k$.
\end{lem}
\begin{proof}
    Let $M$ be a Hamiltonian tour on $R$. Choose $x_0 \in V(R)$ and let $\mathbf{L} = x_0 x_1 \cdots x_{k-1}$ be the word formed by walking along $M$ when starting on $x_0$. Note that $x_i,x_{(i+1\pmod k)}$ are adjacent in $R$ for all $i\in [k]$. Now consider the $K_3$-expansion $\Lambda(R)$ of $R$. We will argue that $\Lambda(R)$ allows for a Soet with respect to the set $V'= \{x_0^{(x_{k-1})},x_0^{(x_1)}, x_1^{(x_0)}, x_1^{(x_2)}, \ldots ,x_{k-1}^{(x_{k-2})},x_{k-1}^{(x_0)}\}$. For all $i \in [k]$ let $v_i$ be the unique vertex adjacent to $x_i$ in $\Lambda(R)$ that is not $x_{(i-1\pmod k)}$ or $x_{(i+1\pmod k)}$. Now consider the following words on $V(\Lambda(R))$. 
\begin{align}
    \mathbf{V} &:= x_0^{x_{k-1}}x_0^{(x_1)}x_1^{(x_0)} x_1^{(x_2)}x_2^{(x_1)}x_2^{(x_3)}\ldots x_{k-1}^{({x_{k-2}})}x_{k-1}^{(x_0)}\label{eq:trailV}\\
    \mathbf{W} &:= x_0^{x_{k-1}}x_0^{(v_0)}x_0^{(x_1)}x_1^{(x_0)}x_1^{(v_1)}x_1^{(x_2)}x_2^{(x_1)}x_2^{(v_2)}x_2^{(x_3)}\ldots x_{k-1}^{({x_{k-2}})}x_{k-1}^{(v_{k-1})}x_{k-1}^{(x_0)}\label{eq:trailW}
\end{align}
    These words describe disjoint trails on $\Lambda(R)$ as illustrated for an example graph in \cref{fig:triangular_expansion_paths}.

    \begin{figure}[H]
        \centering
        \begin{subfigure}{0.4\textwidth}
            \includegraphics[width=\textwidth]{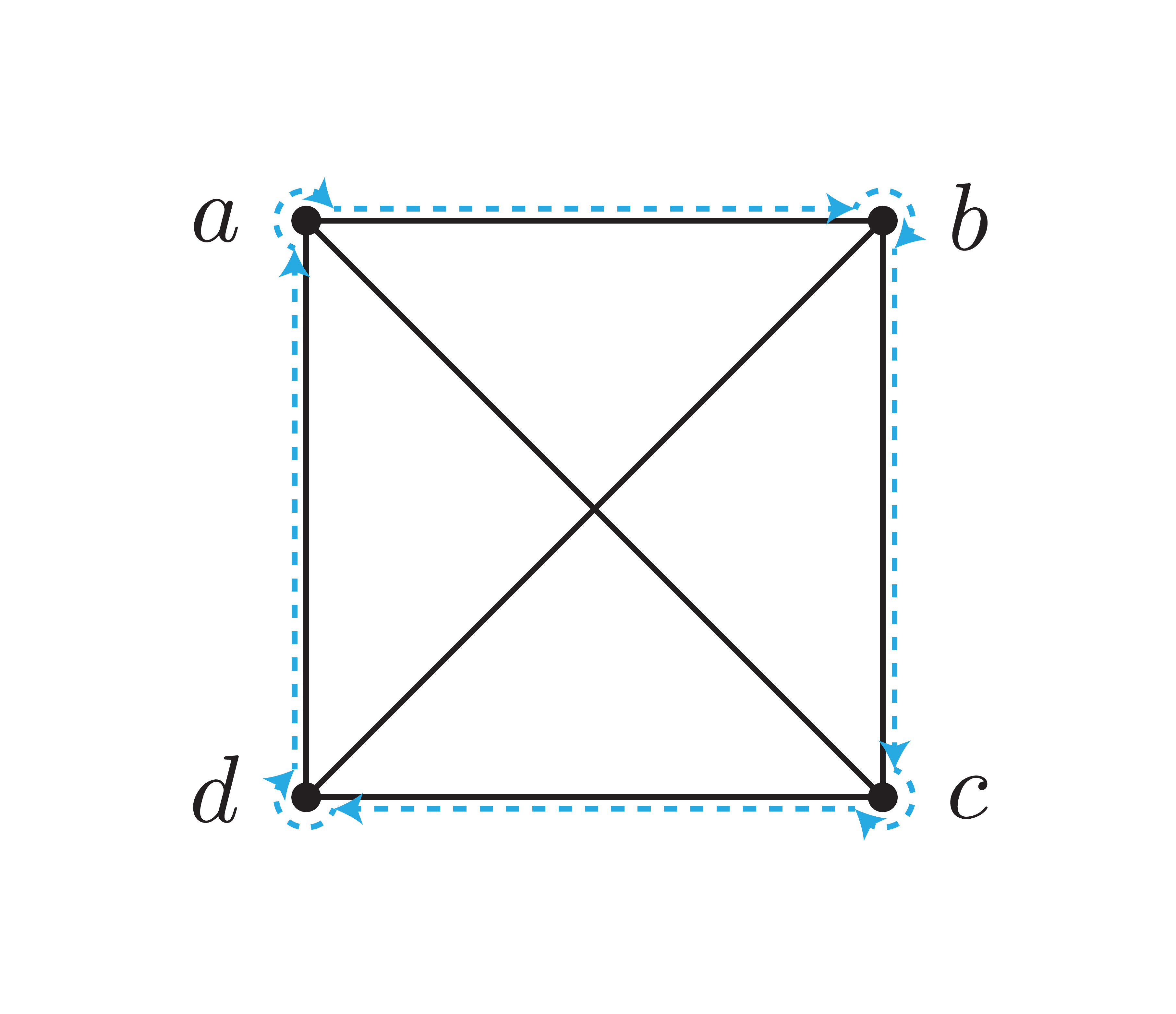}
            \caption{}
        \end{subfigure}
        ~$\quad\quad\quad$
        \begin{subfigure}{0.4\textwidth}
            \includegraphics[width=\textwidth]{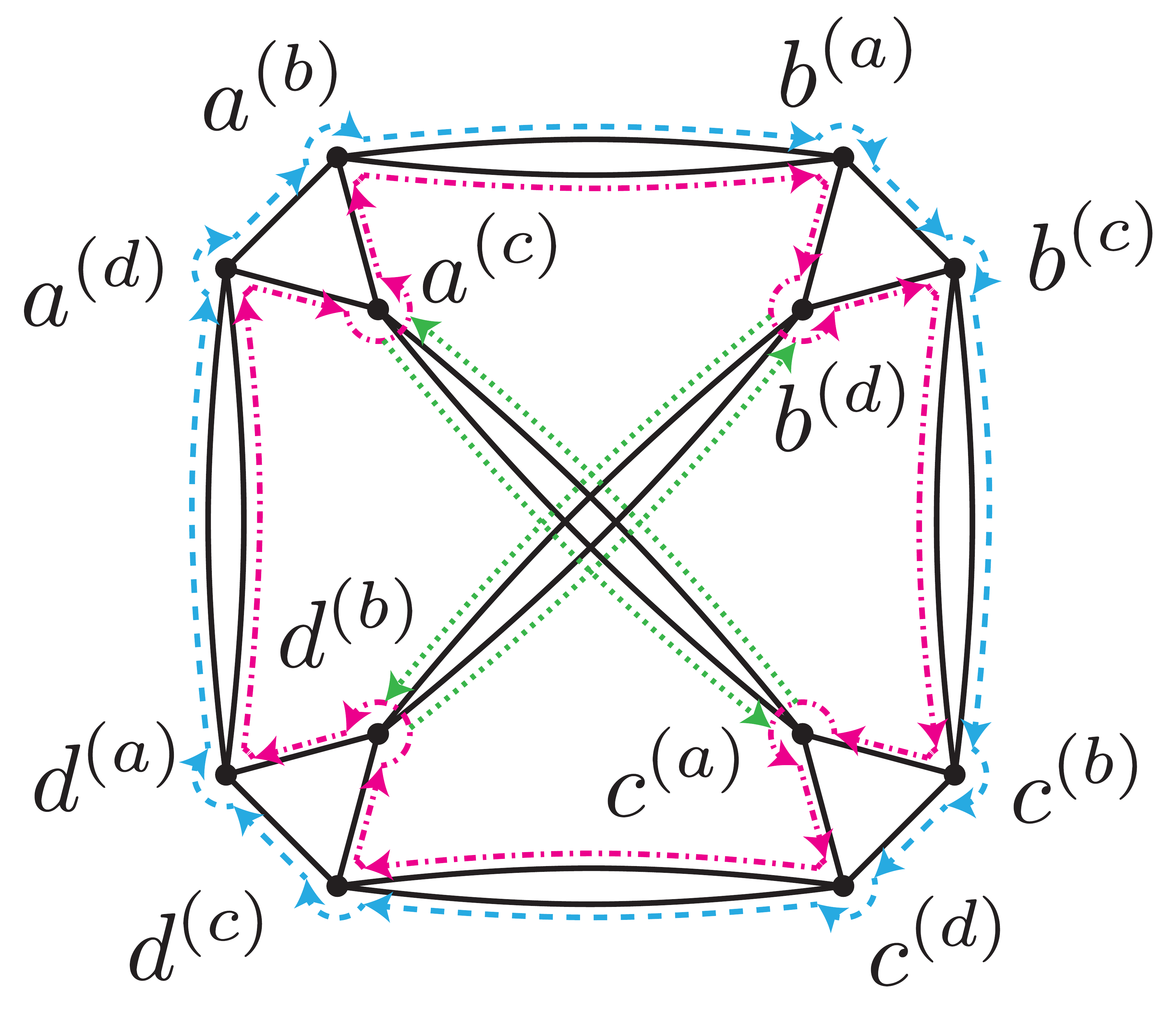}
            \caption{}
        \end{subfigure}
        \caption{Figure showing (a) a Hamiltonian path (blue dashed arrows) on the complete graph on vertices $V=\{a,b,c,d\}$ and (b) the corresponding disjoint trails described by the words $\mathbf{V}$ (blue dashed arrows) and $\mathbf{W}$ (pink dashed-dotted arrows) from \cref{eq:trailV,eq:trailW} on the associated $K_3$-expansion $\Lambda(K_V)$. The edges used to extend the tour to a Eulerian tour as captured by \cref{alg:eul_lifting} are show as green dotted arrows.}
        \label{fig:triangular_expansion_paths}
    \end{figure}
    
    Now consider the word $\mathbf{VW}$. This word describes a trail $U_{VW}$ on $\Lambda(R)$ that visits every vertex in $V'$ exactly twice in the same order. This means $U_{VW}$ is a semi-ordered tour. It is however not Eulerian. To make it Eulerian we have to extend the tour $U_{VW}$ to include all edges in $\Lambda(R)$. Note that these edges are precisely the edges connecting the vertices $x_i^{(v_i)},v_i^{(x_i)}$ for all $i\in [k]$. We can lift $U_{VW}$ to a Eulerian tour by adding vertices to $W$ by the following algorithm.
\begin{algorithm}
\caption{Algorithm for lifting the tour $U_{VW}$ to a Eulerian tour on $\Lambda(R)$}
    \label{alg:eul_lifting}
\begin{algorithmic}
\For{$i \in [k]$}
\If{$x_i^{(v_i)} v_i^{(x_i)} x_i^{(v_i)} \not\subset \mathbf{X'}$}
\State Insert $v_i^{(x_i)} x_i^{(v_i)}$ into $\mathbf{X'}$ right after $x_i^{(v_i)}$
\EndIf
\EndFor
\end{algorithmic}
\end{algorithm}
    It is easy to see that the tour described by $\mathbf{VW}$ after running \cref{alg:eul_lifting} is also Eulerian and is hence a Soet with respect to the set $V'$. This completes the lemma.

\end{proof}

Next we prove a necessary condition (\cref{lem:soet_to_ham}) for the existence of a Soet on a subset $V'$ of the vertices of a $4$-regular graph $F$.

\begin{lem}\label{lem:soet_no_triangle}
    Let $F$ be a $4$ regular graph and $V'\geq4$ be a subset of its vertices. If the fully connected graph $K_3$ is an induced subgraph of $F[V']$ then $F$ does not allow for a Soet with respect to $V'$
\end{lem}

\begin{proof}
    Assume that $F$ has three vertices $V'=\{u,v,w\}$ such that $F[V']=K_3$.
    Let $U$ be a Eulerian tour on $F$ (note that $U$ always exists).
    Assume by contradiction that $U$ is a Soet with respect to $V'$.
    It is easy to see that since $u$ and $v$ are adjacent in $F$ they must also be consecutive in $U$.
    However the same is true for $u$ and $w$ and also $w$ and $v$.
    This means a tour starting at $u$ and must traverse $v$, then $w$, and then immediately $u$ again (up to interchanging $u$ and $w$).
    Since $\{u,v,w\}$ is a strict subset of $V'$ (since by assumption $|V'|\geq 4$) this means that, when starting at $u$, the tour $U$ does not traverse all vertices in $V'$ before returning to $u$. This gives a contradiction with the definition of Soet from which the lemma follows.
\end{proof}

Now we will leverage \cref{lem:soet_no_triangle} to prove that if the $K_3$-expansion $\Lambda(R)$ of a $3$-regular graph $R$ allows for a Soet with respect to a vertex-set $V'$ with $|V'| = 2|V(R)|$ then the graph $R$ must be Hamiltonian.

\begin{lem}\label{lem:soet_to_ham}
Let $R$ be a $3$-regular graph and $\Lambda(R)$ its $K_3$-expansion. If there exists a set $V'\subset V(\Lambda(R))$ with $|V'|=2|V(R)|$ such that $\Lambda(R)$ allows for a Soet with respect to $V'$ then $R$ must be Hamiltonian.
\end{lem}
\begin{proof}
Assume that there exists a subset $V'$ of $V(\Lambda(R))$ with $|V'| = 2|V(R)|$ such that $\Lambda(R)$ allows for a Soet $U$ with respect to $V'$. 

    Note first that since $|V(\Lambda)|=3|V(R)|$ and $|V'|=2|V(R)|$ we must have, by \cref{lem:soet_no_triangle} that $|V[T_u] \cap V'|=2$ for all $u\in V(R)$. This is easiest seen by contradiction. Assume that there exists a $u\in V(R)$ such that $|V[T_u]\cap V'|<2$. Since $V(T_v)\cap V(T_{v'})=\emptyset$ for all $v,v'\in V(R)$, $|V(\Lambda)|=3|V(R)|$, and $|V'|=2|V(R)|$ this implies there must also a exist a $u'\in V(R)$ such that $|V(T_{u'})\cap V'|=3$. This means that $V(T_{u'})\subset V'$. However the induced subgraph $\Lambda(R)[T_{u'}]$ is isomorphic to $K_3$ (this is easily seen from the definition of $K_3$-expansion). By \cref{lem:soet_no_triangle} we must thus conclude that $\Lambda(R)$ does not allow for a Soet with respect to $V'$ leading to a contradiction. Hence we must have that $|V(T_u)\cap V'|=2$ for all $u\in V(R)$.

    Now consider two vertices $x, x'\in V'$ such that $ x,x'$ are consecutive in the Soet $U$.
    Note that, by definition of $\Lambda(R)$, there must exist $w,w'\in V(R)$ such that $x\in T_{w}$ and $x'\in T_{w'}$.
    We will now argue that we must have either $w=w'$ or $w,w'$ are adjacent in $R$.
    We argue this by contradiction.
    Assume thus that $w, w'$ are neither equal nor adjacent in $R$.
    Now consider the maximal sub-word $\mathbf{Y}$ of $m(U)$ associated to $x,x'$.
    Since $w,w'$ are neither equal nor adjacent in $R$, the trail described by the word $\mathbf{Y}$ must pass through a triangle subgraph different from $T_{w}$ and $T_{w'}$, i.e. there exist a vertex $w''\in V(R)$ such that $|\mathbf{Y}\cap V(T_{w''})|\geq 2$.
    However since by construction $|V(T_{w''})\cap V'|=2$ (as shown above) and $|V(T_{w''})|=3$ we must have that $|V'\cap \mathbf{Y}|\geq1$.
    This is however in contradiction with the maximality of the sub-word $\mathbf{Y}$.
    Hence we must have that $w=w'$ or that $(w,w')\in E(R)$. 
     Now consider the word $m(U)$ associated to the Soet $U$ and the induced sub-word $m(U)[V']$. By the above, and the fact that if two vertices in $V'$ are adjacent in $\Lambda(R)$, they must also be consecutive in $U$ (this is a consequence of $U$ being Eulerian and thus having to traverse the edge connecting these vertices), we have that $m(U)[V']$ must be of the form
\begin{equation}
    m(U)[V'] = x_0x'_0x_1x'_1x_2x'_2 \ldots x_{k-1}x'_{k-1}x_0x'_0 \ldots x_{k-1}x'_{k-1}
\end{equation}
    where $x_i,x'_i\in T_{w_i}$ and $\{w_1, \ldots w_k\} = V(R)$ and moreover that $(w_i,w_{i+1})\in E(R)$ for all $i \in [k]$ and also $(w_k,w_{0})\in E(R)$. This immediately implies that the word $\mathbf{M} = w_1 w_2\ldots w_k$ describes a Hamiltonian tour on $R$, and hence that $R$ is Hamiltonian.

\end{proof}

Since \cref{lem:soet_to_ham} and \cref{lem:ham_to_soet} provide necessary and sufficient conditions for a $3$-regular graph being Hamiltonian in terms of whether a $K_3$-expanded graph allows for a Soet, we can now easily prove the hardness of \cref{prob:Soet}.

\begin{thm}
\SOET\ is NP-Hard.
\end{thm}
\begin{proof}
Let $R$ be an instance of \CH, that is, a $3$-regular graph on $k$ vertices. From $R$ we can construct the $4$-regular $K_3$-expansion $\Lambda(R)$. Note that this can be done in poly-time in $k$. Now note that $(\Lambda(R), 2k)$ is an instance of \SOET. If $R$ is a YES instance of \CH, that is, $R$ is Hamiltonian, then by \cref{lem:ham_to_soet} we have that $(\Lambda(R), 2k)$ is a YES instance of \SOET. On the other hand, if $(\Lambda(R), 2k)$ is a YES instance of \SOET, then $R$ is a YES instance of \CH\ by \cref{lem:soet_to_ham}. By contra-position this means that if $R$ is a NO instance of \CH, then $(\Lambda(R), 2k)$ is a NO instance of \SOET. This means \CH\ is Karp-reducible to Soet. Since \CH\ is NP-complete~\cite{karp1972reducibility}, this implies that Soet\ is NP-hard.
\end{proof}

\subsection{ISO-VERTEXMINOR is NP-Hard}\label{ssec:soet_to_svm}
Note first that \SVM\ trivially reduces to \VM, as it is a strict sub-problem. This means that if \SVM\ is NP-hard then so is \VM. 
In this section we show that the \SOET\ reduces to \VM. For this we will make use of the properties of circle graphs, discussed in \cref{sec:prel}.

\Cref{cor:reg_to_circle} states that a $4$-regular multi-graph $F$ allows for a Soet\ with respect to a subset of its vertices $V'\subseteq V(F)$ if and only if an alternance graph $\mathcal{A}(U)$ (which is a circle graph), induced by some Eulerian tour on $F$, has a star graph $S_{V'}$ on $V'$ as a vertex-minor.

Since circle graphs are a subset of all simple graphs we can then decide whether a $4$-regular graph $F$ allows for a Soet~with respect to some subset $V'$ of its vertices by constructing the circle graph induced by an Eulerian tour on $F$ and checking whether it has a star-vertex-minor on the vertex set $V'$. This leads to the following theorem.

\begin{thm}\label{thm:red_Soet_SVM}
The decision problem \SOET\ reduces to \SVM.
\end{thm}
\begin{proof}
    Let $(F,k)$ be an instance of \SOET, where $F$ is a $4$-regular multi-graph and $k\leq |V(F)|$ some integer. Also let $G$ be a circle graph induced by some Eulerian tour $U$ on $F$. From \cref{cor:reg_to_circle} we see that $G$ has $S_{V'}$ as a vertex-minor for some subset of vertices $V'$ of $G$ if and only if $F$ allows for a Soet~with respect to this vertex set $V'$. Since an Eulerian tour $U$ can be found in polynomial time~\cite{Fleury1883} and since $G$ can be efficiently constructed given $U$~\cite{dahlberg2018transform}, considering the case of $|V'|=k$ concludes the reduction.
\end{proof}

\subsection{ISO-VERTEXMINOR is in NP}
Next we argue that the problem \VM\ is in NP. This just follows from the fact that the non-isomorphic vertex-minor problem is in NP. 

\begin{thm}\label{thm:vm_in_np}
The decision problem \VM\ is in NP.
\end{thm}
\begin{proof}
    From~\cite{dahlberg2018transform} we know that there exists a polynomial-length witness for the problem of deciding if a labeled graph $G$ has a vertex-minor equal to another graph $H$ on some fixed subset of its vertices.
    Since $\mathrm{GRAPHISOMORPHISM}$ is in NP we can construct a polynomial-length for \VM, i.e. to decide if $G$ has a vertex-minor isomorphic to $H$.
    We thus conclude that \VM\ is in NP.
\end{proof}

\section*{Acknowledgements}
AD, JH and SW were supported by an ERC Starting grant, and NWO VIDI grant, and Zwaartekracht QSC.

\bibliographystyle{unsrt}
\bibliography{iso-vm-library}

\end{document}